\documentclass[12pt,a4paper]{article}
\usepackage{amssymb,amsmath,amsthm}

\newcommand\cA{{\mathcal A}}

\newcommand\cC{{\mathcal C}}

\newcommand\cF{{\mathcal F}}
\newcommand\cG{{\mathcal G}}

\theoremstyle{plain}
\newtheorem{theorem}{Theorem}[section]
\newtheorem{lemma}[theorem]{Lemma}
\newtheorem{corollary}[theorem]{Corollary}
\newtheorem{conjecture}[theorem]{Conjecture}

\theoremstyle{definition}

\newtheorem{claim}[theorem]{Claim}

\newcommand\lref[1]{Lemma~\ref{lem:#1}}
\newcommand\tref[1]{Theorem~\ref{thm:#1}}
\newcommand\cref[1]{Corollary~\ref{cor:#1}}
\newcommand\clref[1]{Claim~\ref{clm:#1}}

\newcommand\cjref[1]{Conjecture~\ref{conj:#1}}

\title{Families that remain $k$-Sperner even after omitting an element of their ground set}

\author{Bal\'azs Patk\'os\thanks{Alfr\'ed R\'enyi Institute of Mathematics, P.O.B. 127, Budapest H-1364, Hungary. Email: patkos@renyi.hu. Research supported by
    Hungarian National Scientific Fund, grant number: PD-83586 and the J\'anos Bolyai Research Scholarship of the Hungarian Academy of Sciences. Research was done while the author was visiting Zhejiang Normal University, Jinhua, China.}}

\begin{document}
\maketitle

\begin{abstract}
A family $\cF\subseteq 2^{[n]}$ of sets is said to be $l$-trace $k$-Sperner if for any $l$-subset $L \subset [n]$ the family $\cF|_L=\{F|_L:F \in \cF\}=\{F \cap L: F \in \cF\}$ is $k$-Sperner, i.e. does not contain any chain of length $k+1$. The maximum size that an $l$-trace $k$-Sperner family $\cF \subseteq 2^{[n]}$ can have is denoted by $f(n,k,l)$.
For pairs of integers $l<k$, if in a family $\cG$ every pair of sets satisfies $||G_1|-|G_2||<k-l$, then $\cG$ possesses the $(n-l)$-trace $k$-Sperner property. Among such families, the largest one is $\cF_0=\{F\in 2^{[n]}: \lfloor \frac{n-(k-l)}{2}\rfloor+1 \le |F| \le \lfloor \frac{n-(k-l)}{2}\rfloor +k-l\}$ and also $\cF'_0=\{F\in 2^{[n]}: \lfloor \frac{n-(k-l)}{2}\rfloor \le |F| \le \lfloor \frac{n-(k-l)}{2}\rfloor +k-l-1\}$ if $n-(k-l)$ is even.
In an earlier paper, we proved that this is asymptotically optimal for all pair of integers $l<k$, i.e. $f(n,k,n-l)=(1+o(1))|\cF_0|$. In this paper we consider the case when $l=1$, $k\ge 2$, and prove that $f(n,k,n-1)=|\cF_0|$ provided $n$ is large enough. We also prove that the unique $(n-1)$-trace $k$-Sperner family with size $f(n,k,n-1)$ is $\cF_0$ and also $\cF'_0$ when $n+k$ is odd.
\end{abstract}


\section{Introduction}

We use standard notation. The set of the first $n$ positive integers is denoted by $[n]$. For a set $X$ the family of all subsets of $X$, all $i$-subsets of $X$, all subsets of $S$ of size at most $i$, all subsets of $S$ of size at least $i$ are denoted by $2^X, \binom{X}{i}, \binom{X}{\le i}, \binom{X}{\ge i}$, respectively. A chain of length $k$ is a family of $k$ sets satisfying $F_1 \subset F_2 \subset ... \subset F_k$. A maximal chain $\cC \subseteq 2^{[n]}$ is a chain of length $n+1$.

Let $\Sigma(n,m)=\sum_{i=1}^m\binom{n}{\lfloor \frac{n-m}{2}\rfloor+i}$ denote the sum of the $m$ largest binomial coefficients of order $n$.

A typical problem in extremal set system theory is to determine how many sets a family $\cF \subseteq 2^{[n]}$ may contain if it satisfies some prescribed property. As one of the first such result, Erd\H os \cite{E} in 1945 proved that if a family $\cF \subseteq 2^{[n]}$ does not contain any chain of length $k+1$ (families with this property are called \textit{$k$-Sperner families}), then the size of $\cF$ cannot exceed $\Sigma(n,k)$ and the only $k$-Sperner family of this size is $\{F\in 2^{[n]}: \lfloor \frac{n-k}{2}\rfloor+1 \le |F| \le \lfloor \frac{n-k)}{2}\rfloor +k\}$ if $n+k$ is odd, and also $\{F\in 2^{[n]}: \lfloor \frac{n-k}{2}\rfloor \le |F| \le \lfloor \frac{n-k}{2}\rfloor +k-1\}$ if $n+k$ is even. The case $k=1$ was proved by Sperner \cite{S} in 1928.

The \textit{trace} of a set $F$ on another set $X$ is $F \cap X$ and is denoted by $F|_X$. The trace of a family $\cF$ on $X$ is the family of traces $\cF_X=\{F|_X|F \in \cF\}$. The fundamental result about traces of families, known as Sauer-lemma, was proved in the early 70's independently by Sauer \cite{Sa}, Shelah \cite{Sh}, and Vapnik and Chervonenkis \cite{VC} and states that if for a family $\cF \subseteq 2^{[n]}$ there exists no set $K$ of size $k$ such that $\cF|_K=2^{K}$ (i.e. all subsets of $K$ appear as a trace of a set in $\cF$), then $|\cF| \le \sum_{i=0}^{k-1} \binom{n}{i}$ holds. This bound is sharp as shown by the families $\binom{[n]}{\le k-1}$ and $\binom{[n]}{\ge n-k+1}$, but there are lots of other families of this size satisfying the condition of the Sauer-lemma.

One way to make sure that a family $\cF$ satisfies the condition of the Sauer-lemma is to prescribe not to contain any maximal chain as trace in any $k$-subset $K$ of $[n]$. This observation leads to the following notion introduced in \cite{P1}: a family $\cF$ is said to be \textit{$l$-trace $k$-Sperner} if for any $l$-set $L$ the trace $\cF|_L$ is $k$-Sperner. In \cite{P1}, it was proved that if $k \le l$ and $n$ is large enough, then the maximum size $f(n,k,l)$ that an $l$-trace $k$-Sperner family $\cF \subseteq 2^{[n]}$ can have is $\sum_{i=0}^{k-1}\binom{n}{i}$ and the only $l$-trace $k$-Sperner families of this size are $\binom{[n]}{\le k-1}$ and $\binom{[n]}{\ge n-k+1}$, i.e. if the condition of the Sauer-lemma is strengthened to the $l$-trace $k$-Sperner property, then the uniqueness of the 'trivial' extremal families $\binom{[n]}{\le k-1}$ and $\binom{[n]}{\ge n-k+1}$ holds.

The situation is entirely different if for fixed $k$ and $l$ with $k>l$ and $n$ large enough, we consider $(n-l)$-trace $k$-Sperner families, i.e. families of which the traces form a $k$-Sperner family no matter which $l$ elements of the ground set we omit. For any $l$-set $L$ and $G \in 2^{[n]}$ we have $|G|-l \le |G|_{[n]\setminus L}| \le |G|$ and thus if in a family $\cG$ every pair of sets satisfies $||G_1|-|G_2||<k-l$, then $\cG$ possesses the $(n-l)$-trace $k$-Sperner property. We obtain the largest such family if we take $\cG_0=\{G\in 2^{[n]}: \lfloor \frac{n-(k-l)}{2}\rfloor+1 \le |G| \le \lfloor \frac{n-(k-l)}{2}\rfloor +k-l\}$ and also $\cG'_0=\{G\in 2^{[n]}: \lfloor \frac{n-(k-l)}{2}\rfloor \le |G| \le \lfloor \frac{n-(k-l)}{2}\rfloor +k-l-1\}$ if $n-(k-l)$ is even. The size of $\cG_0$ is $\Sigma(n,k-l)$. In \cite{P1} and \cite{P2} we conjectured that the families $\cG_0$ and $\cG'_0$ are optimal.

\begin{conjecture}
\label{conj:precise} Let $k$ and $l$ be positive integers with $l<k$. Then there exists $n_0=n_0(k,l)$ such that if $n \ge n_0$, then $f(n,k,n-l)= \Sigma(n,k-l)$ holds. 
\end{conjecture}

\cjref{precise} was proved asymptotically in \cite{P2} (the case $l=1$, $k=2$ was already proved in \cite{P1}).

\begin{theorem} \cite{P1}
\label{thm:asy} Let $k$ and $l$ be positive integers with $l<k$. Then $f(n,k,n-l)=(1+O(\frac{1}{n^{2/3}})) \Sigma(n,k-l)$ holds.
\end{theorem}

Our main result verifies \cjref{precise} for $l=1, k\ge 2$ and it also describes the extremal family.

\begin{theorem}
\label{thm:l1} Let $k\ge 2$ be an integer. Then there exists $n_0=n_0(k)$ such that if $n \ge n_0$ and $\cF \subseteq 2^{[n]}$ is an $(n-1)$-trace $k$-Sperner family, then $|\cF| \le \Sigma(n,k-1)$. Furthermore, equality holds if and only if $\cF$ is the family $\cG_0$ and when $n+k$ is odd, then also if $\cF$ is the family $\cG'_0$.
\end{theorem}

Note that Sperner's result follows from the case $k=2$ as if $\cF \subseteq 2^{[n]}$ is Sperner, then it is $(n-1)$-trace 2-Sperner. Indeed, if $x \in [n]$ and $F_1,F_2,F_3 \in \cF$ were such that $F_1|_{[n]-x} \subsetneq F_2|_{[n]-x} \subsetneq F_3|_{[n]-x}$ would hold, then there would exist $1 \le i<j\le 3$ with $x \in F_i,F_j$ or $x \notin F_i,F_j$ and thus $F_i \subsetneq F_j$ would contradict the Sperner property of $\cF$. In general it is not true that a $k$-Sperner family possesses the $(n-1)$-trace $(k-1)$-Sperner property, but the largest such family does. \tref{l1} states that no other $(n-1)$-trace $(k-1)$-Sperner family can have larger size.

In the proof of \tref{l1} we will need the next result which follows from the Corollary after Theorem 7 in \cite{P1}.

\begin{theorem}
$$f(n,n-1,1)=O\left(\frac{1}{n}\binom{n}{\lfloor n/2\rfloor}\right) =O\left(\frac{1}{n}\Sigma(n-|F_0|,1)\right).$$
\end{theorem}

\textbf{Notation.} The complement of a set $F$ will be denoted by $\overline{F}$. For convenience, we will write $F+z$ instead of $F \cup \{z\}$ and $F-z$ instead of $F\setminus \{z\}$. Also, if $\sigma$ is a permutation of elements of an $m$-set $M$, then we will think of $\sigma$ as an ordering of the elements of $M$ and write $\sigma_1,\sigma_2,...,\sigma_m$ for the first, second, ... , $m$th element in the ordering. The \textit{index} of an element $x$ in the permutation $\sigma$ is the integer $i$ for which $\sigma_i=x$ and will be denoted by $ind(x)$. The set of permutations of $[n]$ is denoted by $S_n$.
Let $\pi \in S_n$, $a,b,c,d\in [n]$ and $a,b,c,d \notin Y\subset [n]$. Then we will write $\pi$ as
$$........................a.............bcYd......................$$
to denote the fact that $ind(a)< ind(b)=ind(c)-1< ind(d)-2$ and exactly those elements $y$ belong to $Y$ for which $ind(c)<ind(y)<ind(d)$ holds.
Furthermore, if we are interested in the relation of the indices of elements of two not necessarily disjoint subsets $A=\{a,b,c,d\}$ and $X=\{x,y,z\}$ of $[n]$, then the permutation $\pi$ is written as
\[
......................
\begin{array}{ccccc}
. & . & x & y & z \\
a & b & c & . & d
\end{array}
........................
\]
to denote $ind(a)=ind(b)-1=ind(c)-2=ind(x)-2=ind(y)-3=ind(z)-4=ind(d)-4$ and thus $x=c, z=d$ and $|A\cap X|=2$ hold.

\section{Proof of \tref{l1}}

Let $\cF$ be an $(n-1)$-trace $k$-Sperner family and let $\cC \subset 2^{[n]}$ be a maximal chain. Note that if $n>k$, then $|\cF \cap \cC| \le k$ holds. Indeed, if $G_1 \subsetneq G_2 \subsetneq ... \subsetneq G_{k+1} \subseteq [n]$ holds, then either $G_1$ or $[n]\setminus G_{k+1}$ is non-empty or at least one of the sets $G_i \setminus G_{i-1}$ contains two elements. When omitting an element from one such set, the traces of the $G_i$'s would still form a chain of length $k+1$.

Let $c^-,c, c^+$ denote the number of maximal chains $\cC \subseteq 2^{[n]}$ such that $|\cF \cap \cC|$ is less than $k-1$, exactly $k-1$, exactly $k$, respectively. By the above observation, we know that $c^-+c+c^+=n!$ holds. Let $\cF^k=\{(F_1,F_2,...,F_k):F_1 \subsetneq F_2 \subsetneq ... \subsetneq F_k, F_i \in \cF\}$ denote the set of $k$-chains in $\cF$. The main step of the proof of \tref{l1} is the following lemma which states that on average a maximal chain contains at most $k-1$ sets from an $(n-1)$-trace $k$-Sperner family.

\begin{lemma}
\label{lem:chains}
Let $\cF$ be an $(n-1)$-trace $k$-Sperner family such that $4 \le |F| \le n-1$ holds for all $F \in \cF$. Then the inequality $c^- \ge c^+$ holds. Moreover, if there exists a $k$-chain $(F_1,F_2,...,F_k) \in \cF^k$ with $5\le |F_1|$ and $|F_{i+1} \setminus F_i|=1$ for all $i=1,2,...,k-1$, then $c^- > c^+$ holds.
\end{lemma}

\begin{proof}
Let $\cF$ be an $(n-1)$-trace $k$-Sperner family. We will say that $(F_1,F_2,...,F_k) \in \cF^k$ is of type I if  $|F_{i+1}\setminus F_i|=1$ holds for all $1 \le i \le k-1$ and $(F_1,F_2,...,F_k) \in \cF^k$ is of type II-$\ell$ if  $|F_{i+1}\setminus F_i|=1$ holds for all $1 \le i \le \ell-1$ and $|F_{\ell+1} \setminus F_{\ell}| \ge 2$. Note that $$c^+=\sum_{(F_1,F_2,...,F_k)\in \cF^k}(n-|F_k|)!\prod_{i=1}^k(|F_i|-|F_{i-1}|)!$$ where $|F_0|$ is defined to be 0.

Let $(F_1,F_2,..,F_k) \in \cF^k$ be a $k$-chain of type I and $x \in F_1, z \notin F_k$. Then let $\cC(x,z,F_1,F_2,...,F_k)$ denote the set of those maximal chains that contain $(F_i-x)+z$ for all $1 \le i \le k$ and $F_k+z$. Note that if $\{x,z\}\neq \{x',z'\}$, then $\cC(x,z,F_1,F_2,...,F_k) \cap \cC(x',z',F_1,F_2,...,F_k)=\emptyset$ as the sets of size $|F_1|$ in $\cC(x,z,F_1,F_2,...,F_k)$ and $\cC(x',z',F_1,F_2,...,F_k)$ are $(F_1-x)+z$ and $(F_1-x')+z'$, respectively. Therefore writing $\cC(F_1,F_2,...,F_k)=\bigcup_{x\in F_1, z \notin F_k}\cC(x,z,F_1,F_2,...,F_k)$ we have
$$|\cC(F_1,F_2,...,F_k)|=\sum_{x \in F_1, z \notin F_k}|\cC(x,z,F_1,F_2,...,F_k)|=$$
$$|F_1|\cdot(n-|F_k|)\cdot|F_1|!(n-|F_k|-1)!=|F_1|\cdot(n-|F_k|)!\prod_{i=1}^k(|F_i|-|F_{i-1}|)!$$

\begin{claim}
\label{clm:t1good}Let $(F_1,F_2,..,F_k) \in \cF^k$ be a $k$-chain of type I and $x \in F_1, z \notin F_k$. Then for any $\cC \in \cC(x,z,F_1,F_2,...,F_k)$ we have $|\cC \cap \cF| \le k-2$. 
\end{claim}

\begin{proof}[Proof of Claim]
Let $\cC \in \cC(x,z,F_1,F_2,...,F_k)$ be a maximal chain. If $C \in \cC$ with $C \subseteq (F_1 -x)+z$, then $C \notin \cF$ as $C|_{[n]-z} \subsetneq F_1|_{[n]-z} \subsetneq ... \subsetneq F_k|_{[n]-z}$ would form a chain of length $k+1$ in $\cF|_{[n]-z}$. Also, if $C \supseteq F_k+z$, then $C \notin \cF$ as $F_1,F_2,...,F_k,C$ would form a $(k+1)$-chain even without omitting any element of the ground set $[n]$.

Finally, it cannot happen that $F_i'=(F_i -x)+z \in \cF$ holds for all $2 \le i \le k$ as then $F_1|_{[n]-x} \subsetneq F_2|_{[n]-x} \subsetneq F_2'|_{[n]-x} \subsetneq ... \subsetneq F_n'|_{[n]-x}$ would be a $(k+1)$-chain in $\cF|_{[n]-x}$.
\end{proof}

Let $(F_1,F_2,..,F_k) \in \cF^k$ be a $k$-chain of type II-$\ell$ with $1 \le \ell$, $x \in F_1, z \notin F_k$ and $\sigma$ be a permutation of $F_{\ell+1} \setminus F_{\ell}$. Then writing $y_i$ for the unique element of $F_{i+1}\setminus F_i$ for $1 \le i \le \ell-1$ and $m=|F_{\ell+1}\setminus F_{\ell}|$, let $\cC(x,z,\sigma,F_1,F_2,...,F_k)$ denote the set of those maximal chains that contain all sets from 
$$\cA_1=\{F_i -x:1 \le i \le \ell\},$$ 
$$\cA_2=\{(F_{\ell}-x) +\sigma_1,F_{\ell}+\sigma_1, (F_{\ell}+\sigma_1)+\sigma_2,...,F_{\ell+1}-\sigma_{m}, (F_{\ell+1}-\sigma_{m}) +z\}, $$
$$\cA_3=\{F_j+z:\ell +1 \le j \le k\}.$$  
Again, it is easy to see that for a fixed $k$-chain $(F_1,F_2,...,F_k)$ of type II the sets of chains $\cC(x,z,\sigma,F_1,F_2,...,F_k)$ are pairwise disjoint, therefore writing $\cC(F_1,F_2,...,F_k)=$ \newline $\bigcup_{x\in F_1, z \notin F_k, \sigma}\cC(x,z,\sigma,F_1,F_2,...,F_k)$ we have
$$|\cC(F_1,F_2,...,F_k)|=\sum_{x \in F_1, z \notin F_k,\sigma}|\cC(x,z,\sigma,F_1,F_2,...,F_k)|=$$
$$|F_1|\cdot(n-|F_k|)\cdot (|F_1|-1)!(n-|F_k|-1)!\prod_{i=2}^k(|F_i|-|F_{i-1}|)!=(n-|F_k|)!\prod_{i=1}^k(|F_i|-|F_{i-1}|)!$$

\begin{claim}
\label{clm:t2+good}Let $(F_1,F_2,..,F_k) \in \cF^k$ be a $k$-chain of type II-$\ell$ with $2\le \ell$, $x \in F_1, z \notin F_k$ and $\sigma$ as above. Then for any $\cC \in \cC(x,z,\sigma,F_1,F_2,...,F_k)$ we have $|\cC \cap \cF| \le k-2$. 
\end{claim}

\begin{proof}[Proof of Claim]
Let $\cC \in \cC(x,z,\sigma,F_1,F_2,...,F_k)$ be a maximal chain. If $C \in \cC$ with $C \subseteq F_1 -x$, then $C \notin \cF$, as then $C \subsetneq F_1$ would hold and $\cF$ would contain a $(k+1)$-chain.

Note also that $\cA_2 \cap \cC=\emptyset$. Indeed, for any $1 \le j < |F_{\ell+1}\setminus F_{\ell}|$ and writing $A_j=(...((F_\ell+\sigma_1)+\sigma_2)+...)+\sigma_j$ we have $F_\ell\subsetneq A_j \subsetneq F_{\ell+1}$ and thus $A_j$ and the $F_i$'s would form a $(k+1)$-chain. Also, the traces of $A_0=(F_{\ell}-x) +\sigma_1$ and the $F_i$'s would form a $(k+1)$-chain in $\cF|_{[n]-x}$ and the traces of $A_m=(F_{\ell+1}-\sigma_m)+z$ and the $F_i$'s would form a $(k+1)$-chain in $\cF|_{[n]-z}$. (These two statements use the fact that $m=|F_{\ell+1}\setminus F_{\ell}| \ge 2$.)

We obtained that $\cC \cap \cF$ might contain at most $\ell-1$ sets of $\cA_1$ and some sets $C_1,...,C_t$ containing $F_{\ell+1}+z$. Observe that $t \le k-\ell-1$ as otherwise $F_1,...,F_{\ell+1}$ together with $C_1, ...,C_t$ would form a chain of length at least $k+1$. Therefore $|\cC \cap \cF|\le \ell-1+k-\ell-1=k-2$ as stated by the Claim.
\end{proof}

We introduce further notation. First note that maximal chains are in a one-to-one correspondence with permutations of the ground set $[n]$ as with any maximal chain $\cC=\{F_0,F_1,...,F_n\}$ one can associate the permutation $\pi=\pi(\cC)$ such that $\pi_i=F_i\setminus F_{i-1}$. The set of permutations corresponding to maximal chains in $\cC(x,z,\sigma,F_1,F_2,...,F_k)$, $\cC(x,z,F_1,F_2,...,F_k)$, $\cC(F_1,F_2,...,F_k)$ will be denoted $\pi(x,z,\sigma,F_1,F_2,...,F_k)$, $\pi(x,z,F_1,F_2,...,F_k)$, $\pi(F_1,F_2,...,F_k)$, respectively. Permutations $\pi \in S_n$ belonging to $\pi(x,z,F_1,F_2,...,F_k)$ for a $k$-chain of type I look like this

$$ ..................z.......................y_1y_2...y_{k-1}x.....................................\hskip 0.5truecm .$$
Knowing $\pi$, $ind(x)$ and $ind(z)$ we are able to recover the $F_i$'s by 
$$F_i=(\{\pi_j: j \le ind(x)-k+i-1\}-\pi_{ind(z)})+\pi_{ind(x)}.$$

Permutations $\pi \in S_n$ belonging to $\pi(x,z,\sigma,F_1,F_2,...,F_k)$ for a $k$-chain of type II-$\ell$ look like this
$$ .....................y_1y_2....y_{\ell-1}\sigma_1x\sigma_2....\sigma_{m-1}z\sigma_mY_{\ell+1}....Y_{k-1}................\hskip 0.5truecm .$$
Just as for $k$-chains of type I, we are able to recover the $F_i$'s if we know $\pi$, $ind(y_i)$ $i=1,...,\ell-1$, $ind(x)$, $ind(z)$, $\max\{ind(y):y \in F_{j+1} \setminus F_{j}\}$ $j=\ell +1,...,k-1$.

For every $k$-chain $(F_1,F_2,...,F_k) \in \cF^k$ we have defined  a set of maximal chains that contain at most $k-2$ sets from $\cF$. To show that the union of these sets is large we need to prove that there is not much of an overlap among them. We are not able to fully establish such a result, but we manage to prove such statements for subsets of the $\cC(F_1,F_2,...,F_k)$'s. For every $k$-chain $(F_1,F_2,...,F_k) \in \cF^k$ of type I, let $\pi^*(x,z,F_1,F_2,...,F_k)=\{\pi \in \pi(x,z,F_1,F_2,...,F_k): ind(z) \le ind(y_1)-2\}$, while for $k$-chains of type II-$\ell$, let $\pi^*(x,z,\sigma,F_1,F_2,...,F_k)$ equal $\pi(x,z,\sigma,F_1,F_2,...,F_k)$. $\cC^*(x,z,F_1,F-2,...,F_k)$ denotes the set of corresponding maximal chains and we obtain  $\pi^*(F_1,F_2,...,F_k)$ and $\cC^*(F_1,F_2,...,F_k)$ by taking union over all $x \in F_1,z \notin F_k$ and $\sigma$ being a permutation of $F_{\ell+1}\setminus F_\ell$. Clearly, for any $k$-chain of type I we have 
\[
|\pi^*(F_1,F_2,...,F_k)|=(|F_1|-2)\cdot(n-|F_k|)!\prod_{i=1}^k(|F_i|-|F_{i-1}|)!
\]

To make the reasoning in the previous paragraph more formal we need the following final notation. For any maximal chain $\cC$ let $s^*(\cC,\cF)$ denote the number of $k$-chains $(F_1,F_2,...,F_k)$ in $\cF^k$ such that $\cC \in \cC^*(F_1,F_2,...,F_k)$ and let $s^*(F_1,F_2,...,F_k)=\max\{s^*(\cC,\cF): \cC$ is a maximal chain with $F_1,F_2,...,F_k \in \cC\}$. By \clref{t1good}, and \clref{t2+good}, we have

$$\label{1} c^-\ge \sum_{(F_1,F_2,...,F_k)\in \cF^k}\frac{|\cC^*(F_1,F_2,...,F_k)|}{|s^*(F_1,F_2,...,F_k)|}.$$

The following two claims will allow us to establish good upper bounds on $s^*(F_1,F_2,...,F_k)$. 

\begin{claim}
\label{clm:t1} For any $\pi \in S_n$ there exists at most two $k$-chains $(F_1,F_2,...,F_k)$ of type I in $\cF^k$ such that $\pi \in \pi(F_1,F_2,...,F_k)$ holds.
\end{claim}

\begin{proof}[Proof of Claim]
Let $(F_1,F_2,...,F_k),(F'_1,F'_2,...,F'_k)$ be $k$-chains of type I and $\pi \in S_n$ such that $\pi \in \pi(x,z,F_1,F_2,...,F_k)\cap \pi(x',z',F'_1,F'_2,...,F'_k)$. We will show that $|ind(x)-ind(x')|=1$ and the claim will follow.
Suppose first that $ind(x)=ind(x')$ and thus $x=x',y_i=y'_i$ hold for all $i=1,2,...,k-1$. Then we must have $ind(z)\neq ind(z')$ as otherwise the two $k$-chains would be the same. But then the traces $F'_1|_{[n]-z'},F_1|_{[n]-z'},...,F_k|_{[n]-z'}$ form a chain of length $k+1$.

Suppose next that $ind(x')+2 \le ind(x)$. Then there is at most one $j$ such that $ind(z)= ind(y'_j)$  and therefore we have $F'_1|_{[n]-z} \subsetneq ... \subsetneq  F'_j|_{[n]-z} \subsetneq F'_{j+2}|_{[n]-z} \subsetneq ...\subsetneq F'_k|_{[n]-z}$. Furthermore if $z \neq x'$, then $x' \in F_{k-1},F_k$ holds. Therefore, in any case, $F'_k|_{[n]-z} \subsetneq F_{k-1}|_{[n]-z} \subsetneq F_k|{[n]-z}$ and thus $F'_1|_{[n]-z} \subsetneq ... \subsetneq F'_j|_{[n]-z} \subsetneq F'_{j+2}|_{[n]-z} \subsetneq ... \subsetneq F'_k|_{[n]-z} \subsetneq F_{k-1}|_{[n]-z} \subsetneq F_k|_{[n]-z}$ would form a chain of length $k+1$.
\end{proof}

\begin{claim}
\label{clm:t2} For any $\pi \in S_n$ and $1 \le \ell \le k-1$ if $\pi \in \pi^*(F_1,F_2,...,F_k)$ for some $k$-chain $(F_1,F_2,...,F_k)$ of type II-$\ell$ in $\cF^k$, then for any other $(F'_1,F'_2,...,F'_k)\in \cF^k$ we have $\pi \notin \pi^*(F'_1,F'_2,...,F'_k)$.
\end{claim}

\begin{proof}[Proof of Claim] 
Let $(F_1,F_2,...,F_k) \in \cF^k$ be a $k$-chain of type II-$\ell$, $x \in F_1, z\notin F_k$, $\sigma$ a permutation of $F_{\ell+1}\setminus F_\ell$ and $\pi \in \pi(x,z,\sigma,F_1,F_2,...,F_k)$. Assume towards a contradiction that $\pi \in \pi^*(F'_1,F'_2,...,F'_k)$ holds for some other $k$-chain in $\cF^k$. We consider cases according to the type of $(F'_1,F'_2,...,F'_k)$. Before starting the case analysis let us introduce the notation $a$ for $\sigma_1$ if $(F_1,F_2,...,F_k)$ is of type II-1 and $a=y_1$ otherwise. Similarly, let $a'$ denote $\sigma'_1$ if $(F'_1,F'_2,...,F'_k)$ is of type II-1 and $a'=y'_1$ otherwise.

\vskip 0.3truecm

\textsc{Case I}: $(F'_1,F'_2,...,F'_k)$ is of type I and $\pi \in \pi^*(x',z',F'_1,F'_2,...,F'_k)$.

Suppose first that $ind(a)\le ind(a')-3$. 
\[
......................
\begin{array}{ccccccccc}
. & . & ... & z' & ... & y'_{\ell-1-h-2} & y'_{\ell-1-h-1} & y'_{\ell-1-h} & ... \\
y_1 & y_2 & ... & y_i & ... & y_{\ell -1} & \sigma_1 & x & ...
\end{array}
........................
\]
Then if $(F_1,F_2,...,F_k)$ is of type II-1, then $\{\pi_j:j\le ind(x)\}=F_1+\sigma_1 \subseteq F'_1\cup z'$ and thus $F_1|_{[n]-z'} \subsetneq F'_1|_{[n]-z'} \subsetneq ...\subsetneq F'_k|_{[n]-z'}$ would contradict the $(n-1)$-trace $k$-Sperner property of $\cF$. If $(F_1,F_2,...,F_k)$ is of type II-$\ell$ with $\ell \ge 2$, then $ind(x)< ind(y'_{\ell-1})$ and thus $F_\ell-z' \subsetneq F'_{\ell-1}=F'_{\ell-1}-z'$ holds. Therefore the traces of $F_1,...,F_\ell,F'_{\ell-1}, F'_{\ell},...,F'_k$ on $[n]-z'$ form a chain and at most two of them may coincide (if $z'=y_i$ for some $i \le \ell-1$) which would still give us a chain of length $k+1$.

Suppose next that $ind(a) = ind(a')-2$. 
\[
......................
\begin{array}{ccccccccc}
z' & ... & . & . & ...  & y'_{\ell-3} & y'_{\ell-2} & y'_{\ell-1} & ... \\
. & ... & y_1 & y_2 & ...  & y_{\ell -1} & \sigma_1 & x & ...
\end{array}
........................
\]
Again, if $(F_1,F_2,...,F_k)$ is of type II-1, then $\{\pi_j:j\le ind(x)\}=F_1+\sigma_1 \subseteq F'_1\cup z'$ and thus $F_1|_{[n]-z'} \subsetneq F'_1|_{[n]-z'} \subsetneq ...\subsetneq F'_k|_{[n]-z'}$ would contradict the $(n-1)$-trace $k$-Sperner property of $\cF$. If $(F_1,F_2,...,F_k)$ is of type II-$\ell$ with $\ell \ge 2$, then $ind(x)= ind(y'_{\ell-1})$ and thus $F_\ell-z' \subsetneq F'_{\ell}=F'_{\ell}-z'$ holds. Note that the condition $\pi \in \pi^*(F'_1,F'_2,...,F'_k)$ implies that $ind(z')<ind(a)$ and thus $z' \in F_i$ for all $i=1,2,...,k$ and the traces $F_i|_{[n]-z'}$ are all distinct. Therefore the traces of $F_1,...,F_\ell, F'_{\ell},...,F'_k$ on $[n]-z'$ form a chain of length $k+1$ which contradicts the $(n-1)$-trace $k$-Sperner property of $\cF$.

Finally suppose that $ind(a)\ge ind(a')-1$. Then the largest index $i$ belonging to an element of $F_k \setminus F_{k-1}$ is strictly larger than $ind(x')$. Indeed, as $(F'_1,F'_2,...,F'_k)$ is of type I we have $ind(x')=ind(a')+k-1$ while since $(F_1,F_2,...,F_k)$ is of type II-$\ell$, we have $ind(x),ind(z)<i$, $|F_{\ell+1}\setminus F_\ell|\ge 2$ and thus $ind(a)+k+2 \le i$. From the inequality $ind(x')<i$ it follows that
\begin{itemize}
\item
if $(F_1,F_2,...,F_k)$ is of type II-$(k-1)$, then $ind(z)=i-1\ge ind(x')$ holds. Therefore  we have $F'_k-x'\subsetneq F_k-x'$ and thus the traces of $F'_1,F'_2,...,F'_k,F_k$ on $[n]-x'$ form a chain of length $k+1$ (note that $x'$ is contained in all $F'_j$'s and therefore omitting $x'$ does not effect their strict containment),
\item
if $(F_1,F_2,...,F_k)$ is of type II-$\ell$ for some $\ell <k-1$, then $x' \in F_{k-1},F_k$ or $x'=z$ holds and thus we have
$F'_k|_{[n]-z} \subsetneq F_{k-1}|_{[n]-z}\subsetneq F_k|_{[n]-z}$, where the first strict containment follows from $z' \notin F'_k, z'\in F_{k-1}$. As omitting $z$ can make at most two of the traces of the $F'_j$'s coincide, $k-1$ of these traces together with $F_{k-1}|_{[n]-z}, F_k|_{[n]-z}$ would still form a chain of length $k+1$ contradicting the $(n-1)$-trace $k$-Sperner property of $\cF$.
\end{itemize}

\vskip 0.3truecm

\textsc{Case II}: $(F_1,F_2,...,F_k)$ and $(F'_1,F'_2,...,F'_k)$ are both of type II-$\ell$.

Suppose $\pi \in \pi^*(x,z,\sigma,F_1,F_2,...,F_k)\cap \pi^*(x',z',\sigma'F'_1,F'_2,...,F'_k)$. Let us first assume that $ind(a)>ind(a')+1$ and thus $ind(x') \le ind(y_{\ell-1})$ and therefore $\sigma'_1,x' \in F_{\ell}$ hold. Consequently, $F'_\ell \subsetneq F_\ell$ holds and hence $F'_1, ...,F'_{\ell},F_{\ell},...,F_k$ form a chain of length $k+1$. 

Assume next that $ind(a)=ind(a')+1$. 
\[
......................
\begin{array}{ccccccc}
. & y_1 & ... & y_{\ell-2} & y'_{\ell-1} & \sigma_1 & ... \\
y'_1 & y'_2 & ...  & y'_{\ell -1} & \sigma'_1 & x' & ...
\end{array}
........................
\]
Then we have $x'=\sigma_1$ and $\sigma'_1\in F_{\ell}$ holds, thus we have $F'_{\ell}|_{[n]-x'} \subsetneq F_{\ell}|_{[n]-x'}$. Therefore, as $x' \in F'_i$ for all $i$ and $x'\in F_j$ for all $j \ge \ell+1$, we obtain a chain $F'_1|_{[n]-x'} \subsetneq ... F'_{\ell}|_{[n]-x'} \subsetneq F_{\ell}|_{[n]-x'} \subsetneq ... \subsetneq F_k|_{[n]-x'}$ of length $k+1$.

Assume finally that $a=a'$ and thus $x=x'$. Suppose first that $ind(z')<ind(z)$. Then $F'_1|_{[n]-z}\subsetneq F'_{\ell+1}|_{[n]-z} \subsetneq F_{\ell+1}|_{[n]-z} \subsetneq ... \subsetneq F_k|_{[n]-z}$ is a chain of length $k+1$.

Suppose then that $x=x',z=z'$. Therefore there must exist $j \ge \ell+2$ such that the largest index of an element in $F_j$ is different, say larger, than the one in $F'_j$. Indeed, otherwise $(F_1,F_2,...,F_k)$ and $(F'_1,F'_2,...,F'_k)$ would be the same. Let $j_0$ be the smallest such number. Then $F'_1,...,F'_{j_0},F_{j_0},...,F_k$ form a chain of length $k+1$.

\vskip 0.3truecm

\textsc{Case III}: $(F_1,F_2,...,F_k)$ is of type II-$\ell$ and $(F'_1,F'_2,...,F'_k)$ is of type II-$\ell'$ for some $\ell' \ne \ell$.

First we claim that $|ind(a)-ind(a')| \le 1$. Suppose not and, say, $ind(a)+2 \le ind(a')$. If $(F_1,F_2,...,F_k)$ is of type II--1, then $F_1 \subsetneq F'_1$ and thus $F_1,F_1,...,F_k$ form a chain of length $k+1$ contradicting the $(n-1)$-trace $k$-Sperner property of $\cF$. If $(F_1,F_2,...,F_k)$ is of type II-$\ell$ for some $\ell \ge 2$, then $F_1-x,F_2-x \subseteq F'_1$. If $x=x'$, then $F_1 \subsetneq F'_1$ and $F_1$ together with all the $F'_j$s form a chain of length $k+1$. If $x \neq x'$, then $F_2-x \subsetneq F'_1$ holds, and as omitting $x$ may make at most two $F'_j$s coincide, the $F_1|_{[n]-x},F_2|_{[n]-x}$ and $k-1$ different $F'_j|_{[n]-x}$'s would form a chain of length $k+1$.

\vskip 0.2truecm

By symmetry wlog we may assume $\ell > \ell'$. By the above we have to consider three cases according to $ind(a)-ind(a')=0,\pm 1$.

\vskip 0.2truecm

\textsc{Subcase III/A}: $ind(a)=ind(a')$.

If $\ell=\ell'+1$, then $x'=\sigma_1$ holds. Therefore omitting $x'$ we have $F'_1|_{[n]-x'} \subsetneq F_1|_{[n]-x'} \subsetneq ... \subsetneq F_k|_{[n]-x'}$ a chain of length $k+1$.

If $\ell\ge \ell'+2$, then $F_1-x\subsetneq F'_1$ and $F'_{\ell'}+\sigma'_1=F_{\ell'+2}-x \subseteq F'_{\ell'+1}$ hold. 
\[
......................
\begin{array}{ccccc}
y_1 & ... & y_{\ell'-1} & y_{\ell'} & y_{\ell'+1}  \\
y'_1 & ... & y'_{\ell' -1} & \sigma'_1 & x' 
\end{array}
........................
\]
Furthermore $F_{\ell'+2}-x = F'_{\ell'+1}$ if and only if $|F'_{\ell'+1} \setminus F'_{\ell'}|=2$ and $x=\sigma'_2$. In this case all $F'_i|_{[n]-x}$'s are different and hence $F_1|_{[n]-x} \subsetneq F'_1|_{[n]-x} \subsetneq ... \subsetneq F'_k|_{[n]-x}$ is a $(k+1)$-chain. Otherwise we would obtain a $(k+1)$-chain by $F_1|_{[n]-x} \subsetneq F'_1|_{[n]-x} \subsetneq ... \subsetneq F'_{\ell'}|_{[n]-x} \subsetneq F_{\ell'+2}|_{[n]-x} \subsetneq F'_{\ell'+1}|_{[n]-x}$ and adding all but at most one further $F'_j|_{[n]-x}$.

\vskip 0.2truecm

\textsc{Subcase III/B}: $ind(a)=ind(a')-1$.

If $\ell=\ell'+1$, then $x=x'$ holds and thus $F_1,F'_1,...F'_k$ form a chain of length $k+1$. If $\ell \ge \ell'+2$, then we have $|F_{\ell'+2}\setminus F_{\ell'+1}|=1$ and $\sigma'_1=y_{\ell'+1}$.
\[
......................
\begin{array}{ccccccc}
y_1 & y_2 & ... & y_{\ell'} & y_{\ell'+1} & ... & x \\
. & y'_1 & ... & y'_{\ell' -1} & \sigma'_1  & ... & .
\end{array}
........................
\]
Therefore $F_{\ell'+2}-x \subsetneq F'_{\ell'+1}$ and as at least $k-\ell'-1$ of $F'_{\ell'+1},...,F'_k$ form a chain even after omitting $x$, the traces of these and those of $F_1,...,F_{\ell'+2}$ on $[n]-x$ would form a chain of length $k+1$.

\vskip 0.2truecm

\textsc{Subcase III/C}: $ind(a)=ind(a')+1$.
\[
......................
\begin{array}{cccccccc}
. &y_1 & ... & y_{\ell'-2} & y_{\ell'-1} & y_{\ell'} & ... & x \\
y'_1 & y'_2 &... & y'_{\ell' -1} & \sigma'_1  & x' & ... & .
\end{array}
........................
\]
Then $\sigma'_1=y_{\ell'-1}$, $x'=y_{\ell'}$ and $ind(x')<ind(x)$. Thus $F'_1|_{[n]-x'} \subsetneq ... F'_{\ell'}|_{[n]-x'} \subsetneq F_{\ell'-1}|{[n]-x'} \subsetneq F_{\ell'}|_{[n]-x'} \subsetneq F_{\ell'+2}|_{[n]-x'} \subsetneq ... \subsetneq F_k|_{[n]-x'}$ would be a chain of length $k+1$ contradicting the $(n-1)$-trace $k$-Sperner property of $\cF$.
\end{proof}

By \clref{t1} and \clref{t2} we obtain $s^*(F_1,F_2,...,F_k) \le 2$ if $(F_1,F_2,...,F_k) \in \cF^k$ is of type I and $s^*(F_1,F_2,...,F_k) = 1$ if $(F_1,F_2,...,F_k) \in \cF^k$ is of type II-$\ell$ for any $1 \le \ell \le k-1$. Therefore by \clref{t1good} and \clref{t2+good} we have
$$c^-\ge \sum_{(F_1,F_2,...,F_k)\in \cF^k}\frac{|\cC^*(F_1,F_2,...,F_k)|}{|s^*(F_1,F_2,...,F_k)|} \ge$$
$$\sum_{\substack{(F_1,F_2,...,F_k)\in \cF^k \\ \textnormal{type I}}}\frac{(|F_1|-2)(n-|F_k|)!}{2}\prod_{i=1}^k(|F_i|-|F_{i-1}|)! +\sum_{\substack{(F_1,F_2,...,F_k)\in \cF^k \\ \textnormal{type II}}}(n-|F_k|)!\prod_{i=1}^k(|F_i|-|F_{i-1}|)! \ge$$
$$ \sum_{(F_1,F_2,...,F_k)\in \cF^k}(n-|F_k|)!\prod_{i=1}^k(|F_i|-|F_{i-1}|)!= c^+$$
The moreover part of the statement follows as if $|F_1|\ge 5$ for some $(F_1,F_2,...,F_k) \in \cF^k$ of type I, then for that particular summand we have at least a $3/2$ fraction more than what we need.
\end{proof}

\begin{corollary}
\label{cor:lym}
Let $\cF$ be an $(n-1)$-trace $k$-Sperner family such that $4 \le |F| \le n-1$ holds for all $F \in \cF$. Then the inequality $$\sum_{F \in \cF}\frac{1}{\binom{n}{|F|}} \le k-1$$ holds. Moreover, if there exists a $k$-chain $(F_1,F_2,...,F_k) \in \cF^k$ of type I with $5\le |F_1|$, then the inequality is strict.
\end{corollary}

\begin{proof}
Let us count the pairs $(F,\cC)$ where $\cC$ is a maximal chain and $F \in \cF\cap \cC$. On the one hand the number of such pairs is $\sum_{F \in \cF}|F|!(n-|F|)!$, on the other hand this is at most $k\cdot c^++(k-1)c+(k-2)c^-$ which is, by \lref{chains}, at most $(k-1)\cdot n!$. Dividing by $n!$ gives the statement and the moreover part follows from the moreover part of \lref{chains}.
\end{proof}

\begin{proof}[Proof of \tref{l1}]
Let $\cF \subseteq 2^{[n]}$ be an $(n-1)$-trace $k$-Sperner family and suppose first that there exists a set $F \in \cF$ with $|F|=0,1,2,3$ or $n$. If $F=\emptyset$ or $[n]$, then $\cF'=\{F \in \cF: 2 \le |F| \le n-2\}$ is $(n-1)$-trace $(k-1)$-Sperner. Indeed, the trace $[n]|_{[n]-x}=[n]-x$ always strictly contains the trace of any set of size at most $n-2$ and the trace $\emptyset|_{[n]-x}=\emptyset$ is always strictly contained in the trace of any set of size at least two. Therefore any $k$-chain in $\cF'|_{[n]-x}$ could be extended to a $(k+1)$-chain in $\cF|_{[n]-x}$. Using \tref{asy} if $k \ge 3$ we obtain that $|\cF'| \le (1+O(\frac{1}{n^{1/3}}))\Sigma(n,k-2)$ and thus $|\cF|\le 2+2n+(1+O(\frac{1}{n^{1/3}}))\Sigma(n,k-2)$, which is strictly less than $\Sigma(n,k-1)$ if $n$ is large enough, while if $k=2$ then \tref{l1} gives that $|\cF'| \le O(\frac{1}{n}\Sigma(n,1))$ and thus $|\cF|<\Sigma(n,1)$ if $n$ is large enough.

Suppose next that there exists $F_0 \in \cF$ with $1 \le |F|\le 3$. For every subset $S$ of $F_0$, let us write $\cF_S=\{F \in \cF: F \cap F_0=S, 5 \le |F| \le n-1\}$. Observe that for any $S \subseteq F_0$ the family $\cF_S|_{[n]\setminus F_0}$ is $(n-|F_0|-1)$-trace $k$-Sperner and $|\cF_S|=|\cF_S|_{[n]\setminus F_0}|$. Therefore, by \tref{asy}, we obtain $|\cF_S| \le (1+o(1))\Sigma(n-|F_0|,k-1)=(1/2^{|F_0|}+o(1))\Sigma(n,k-1)$. Furthermore, the family $\cF_{F_0}|_{[n]\setminus F_0}$ is $(n-|F_0|-1)$-trace $(k-1)$-Sperner. Indeed, as $|F_0| \le 3$ and all sets in $\cF_{F_0}$ have size at least five, we have $F_0|_{[n]-x} \subsetneq F|_{[n]-x}$ for any $F \in \cF_{F_0}$ and $x \in [n]$ and thus adding $F_0|_{[n]-x}$ to any $k$-chain in $\cF_{F_0}|_{[n]-x}$ would create a $(k+1)$-chain in $\cF|_{[n]-x}$. If $k \ge 3$, then \tref{asy} yields $|\cF_{F_0}| \le (1+o(1))\Sigma(n-|F_0|,k-2)=(1/2^{|F_0|}+o(1))\Sigma(n,k-2)=(\frac{k-2}{(k-1)2^{|F_0|}}+o(1))\Sigma(n,k-1)$, and thus
\[
|\cF| \le \sum_{i=1}^4\binom{n}{i}+\sum_{S \subseteq F_0}|\cF_S| <\Sigma(n,k-1),
\]
provided $n$ is large enough. If $k=2$, then \tref{l1} gives that $|\cF_{F_0}| \le O(\frac{1}{n}\Sigma(n-|F_0|,1))$ and thus 
\[
|\cF|\le \sum_{i=1}^4\binom{n}{i}+\left(\frac{2^{|F_0|}-1}{2^{|F_0|}}+o(1)\right)\Sigma(n,1) <\Sigma(n,1)
\] if $n$ is large enough.

\vskip 0.3truecm

We are left with the case when $\cF$ does not contain any set of size 0, 1, 2, 3  or $n$. \cref{lym} yields the statement $|\cF| \le \Sigma(n,k-1)$ and also the uniqueness of the extremal family if $n+k$ is even. If $n+k$ is odd, then $\cF$ must contain only sets of size between $\lfloor \frac{n-k+1}{2} \rfloor$ and $\lceil \frac{n+k-1}{2} \rceil$ and thus cannot contain sets of size four if $\lfloor \frac{n-k+1}{2} \rfloor \ge 5$ holds. By the moreover part of \cref{lym}, $\cF$ cannot contain a $k$-chain of type I with even the smallest set having size at least five. The sizes of the largest and smallest set of a $k$-chain of type II must differ by at least $k$ and thus at least one of them is outside the interval between $\lfloor \frac{n-k+1}{2} \rfloor$ and $\lceil \frac{n+k-1}{2} \rceil$. We obtained that an $(n-1)$-trace $k$-Sperner family of size $\Sigma(n,k-1)$ must be $(k-1)$-Sperner. Then the uniqueness follows from the uniqueness part of Erd\H os's result.
\end{proof}

\end{document}